\documentclass[letterpaper,11pt,reqno]{amsart}

\usepackage{amssymb}
\usepackage{amsmath}
\usepackage{amsthm}
\usepackage{amsfonts}
\usepackage{bbm}
\usepackage{color}

\usepackage{bookmark}
\usepackage{hyperref}
\hypersetup{pdfstartview={FitH}}

\DeclareFontFamily{U}{mathx}{\hyphenchar\font45}
\DeclareFontShape{U}{mathx}{m}{n}{
<5> <6> <7> <8> <9> <10>
<10.95> <12> <14.4> <17.28> <20.74> <24.88>
mathx10
}{}
\DeclareSymbolFont{mathx}{U}{mathx}{m}{n}
\DeclareFontSubstitution{U}{mathx}{m}{n}
\DeclareMathAccent{\widecheck}{0}{mathx}{"71}

\def\R{\mathbb{R}}
\def\Z{\mathbb{Z}}
\def\N{\mathbb{N}}
\def\C{\mathbb{C}}
\def\F{\mathcal{F}}
\renewcommand{\L}{\mathrm{L}}

\newtheorem{theorem}{Theorem}
\newtheorem{lemma}[theorem]{Lemma}

\numberwithin{equation}{section}

\begin{document}
\title {Directional maximal function  along the primes}

\author{Laura Cladek}
\address{Department of Mathematics, UCLA, 520 Portola Plaza, Los Angeles, CA 90095, USA}
\email{cladek@math.ucla.edu}
\author{Polona Durcik}
\address{California Institute of Technology, 1200 E California Blvd, Pasadena, CA 91125, USA}
\email{durcik@caltech.edu}
\author{Ben Krause}
\address{Department of Mathematics, Princeton University, Princeton, NJ 08544, USA}
\email{benkrause2323@gmail.com}
\author{Jos{\'e} Madrid}
\address{Department of Mathematics, UCLA, 520 Portola Plaza, Los Angeles, CA 90095, USA}
\email{jmadrid@math.ucla.edu}

\date{\today}

\begin{abstract}
We study a two-dimensional discrete directional maximal operator along the set of the prime numbers. We show existence of a set of vectors, which are lattice points in a sufficiently large annulus, for which the $\ell^2$ norm of the associated maximal operator with supremum taken over all   large scales grows with an epsilon power in  the number of vectors. This paper  is a follow-up to a prior work on  the discrete directional maximal operator along the integers by the first and third author.
\end{abstract}

\maketitle

\section{Introduction}
A fundamental operator studied in Euclidean harmonic analysis is the directional maximal operator over a finite set of directions in the plane. Given a collection $V\subseteq \mathbb{S}^1$ of unit vectors, one defines the directional maximal operator  by
\begin{align}\label{contest}
M_Vf(x)= \sup_{v\in V}\sup_{r>0}\bigg|\frac{1}{2r}\int_{-r}^rf(x-tv)\,dt\bigg|,
\end{align}
 where $f:\mathbb{R}^2\rightarrow \mathbb{C}$.   The initial interest in this operator lies in its relation to the Kakeya maximal function in the plane, see \cite{St}. 
It was shown by Katz \cite{Katz1} that for any finite set of directions $V$ one has the sharp bound
\begin{align*}
\|M_Vf\|_{\L^2(\R^2)}\leq C (\log |V|)\, \|f\|_{\L^2(\R^2)}
\end{align*}
The  special case when $V$ is uniformly distributed   had been previously settled by Str\"omberg \cite{St}.  

In  \cite{CK}, Cladek and Krause investigated a discrete analogue of  $M_Vf$ restricted to a single dyadic scale. That is, for $V$ a set of lattice points living in an annulus in $\R^2$, a function   $f: \mathbb{Z}^2\to\mathbb{C}$ and $\phi$ a smooth compactly supported function on the real line, they define the single-scale maximal operators $A_{V, k}$ by
$$
A_{V, k}f(x)=\sup_{v\in V}\bigg|\sum_{n\in\mathbb{Z}}f(x-nv)2^{-k}\phi(2^{-k}n)\bigg|
$$
and study its operator norm on $\ell^2$.
By considering bumps adapted to unit scales one sees that the sharp estimate for $k$ close to one is given by $|V|^{1/2}$.  Since the discrete problem is not scale invariant, one may ask if   one can prove an estimate in the spirit of 
 \eqref{contest} for all sufficiently large scales with the threshold  depending on the set of directions. While the answer to this question seems    out of reach of the current techniques, the authors of \cite{CK} construct a collection of vectors for which this is the case. Moreover, their construction is sufficiently robust to   allow for supremum over all large scales. For each $N$ and $\epsilon>0$, they construct a set of vectors $V=V_{N, \epsilon}$ with $|V|=N$ which is  contained  in a large annulus of    radius and thickness about $A$, such that
$$
\big \|\sup_{k\geq  k_V} A_{V, k}f\, \big \|_{\ell^2(\mathbb{Z}^2)}\leq C_\epsilon N^{\epsilon} \|f\|_{\ell^2(\Z^2)},
$$ where $k_V =C_1 \log N$ for some large constant  $C_1$ depending on $\epsilon$. Furthermore, in \cite{CK} the authors also consider discrete maximal operators with Radon-type behavior where the averages are taken along polynomials $P$ with integer coefficients,
\begin{align*}
A^P_{V, k}f(x)=\sup_{v\in V}\bigg|\sum_{n\in\mathbb{Z}}f(x-vP(n))2^{-k}\phi(2^{-k}n)\bigg|.
\end{align*}
 In this case one has the bound 
\begin{align*}
\|A^P_{V, k}f\|_{\ell^2(\mathbb{Z}^2)}\lesssim_{\epsilon}N^{\epsilon} \|f\|_{\ell^2(\Z^2)}
\end{align*}
provided $k\geq k_V$ and $V$ is the collection described   above.
Studying averages along the  polynomial sequences can be viewed as a discrete version of the Furstenberg problem in the plane, see for instance \cite{KaTao}.

In this paper, we study a   problem where the directional averages are taken over {\em prime}, rather than all integer dilates of our vectors $v$.  
Let $k\in \mathbb{N}$ and   $V\subseteq \Z^2$  be  a finite collection of lattice points.  Let $\phi$ be a compactly supported function on the real line.  By $\mathbb{P}$ we denote the set of all prime numbers.  
For a finitely supported function $f:\Z^2\rightarrow \mathbb{C}$
we consider the directional maximal operator  
$$A^{\mathbb{P}}_{V,k}f(x)=\sup_{k\geq k_V}\sup_{v\in V}\Big |   \sum_{p\in \mathbb{P}} f(x-pv)2^{-k}\phi(2^{-k}p)\log p \, \Big |
$$
where $k_V$ is large depending on $|V|$. 
We prove the following bound.
\begin{theorem}
\label{mainthm1}
Let $\epsilon>0$ and $N>0$. There exist constants $C_\epsilon$ and $C_0$ depending on $\epsilon$ so that the    following holds. For any $A\geq N^{C_0}$
there exists a set of vectors $V\subseteq \{x\in \Z^2: \frac{1}{100}A\leq |x| \leq 100A\}$ with   $|V|=N$  so that 
$$\big \|\sup_{k\geq k_V} A^{\mathbb{P}}_{V,k} f\big \|_{\ell^2(\Z^2)} \leq C_\epsilon N^{\epsilon} \|f\|_{\ell^2(\Z^2)},$$
provided  $k_V \geq  N^{C_1}$ for a sufficiently large constant $C_1$ depending on $C_0$.
\end{theorem}
To prove Theorem \ref{mainthm1} one first analyzes the corresponding multipliers, which is standard using the Hardy-Littlewood circle method. See for instance \cite{MTZK}. For each $v$ in our set of directions, the corresponding multiplier may be described in terms of its coordinate in the $v^{\perp}$-axis, and is constant in the orthogonal direction. The corresponding one-dimensional multiplier is supported on disjoint frequency bubbles located near certain rational numbers. This requires an analysis of multiple frequencies simultaneously, which is a key distinguishing feature of many problems in discrete harmonic analysis. In the case of high frequencies, we must moreover analyze multipliers corresponding to different directions $v$ simultaneously, which leads to both number theoretic and geometric considerations. Using harmonic analytic techniques, we reduce the boundedness of the maximal function to an incidence estimate proven in \cite{CK}, which measures overlaps of certain tubes pointing in the directions given by the vectors in our set. We only need to apply this incidence estimate
outside a ball of a fixed radius, which  is sufficiently small to capture only a single tube in each direction.
 Inside this ball one applies transference arguments to extend the continuous result to the discrete result. 

In contrast with the squares, in the case of the primes we are able to 
prove   bounds even with the further introduction of a supremum over scales. The key is that we are able to exploit the near perfect cancellation of the Ramanujan sums, which appear in the corresponding multiplier approximations. This can be seen as a quantitative consequence of the well-known heuristic that the distribution of the primes exhibits a high degree of randomness. Exploiting this cancellation is crucial to our method of proof.  Nevertheless, it would be interesting to see if one can prove $\ell^2$ bounds with the supremum over scales in the case of the squares as well.
 
The problem studied in this paper is related to the works by Bourgain \cite{Bou2}, Wierdl \cite{Wi}, and Mirek, Trojan and Zorin-Kranich \cite{MTZK}. These papers are concerned with maximal inequalities for discrete versions of the Hardy-Littlewood maximal operator along the primes, motivated  by pointwise convergence of  ergodic averages along  the primes.
The paper \cite{MTZK} shows   variational estimates for discrete  maximal operators and for maximal truncation of singular integrals along the primes. It would be interesting to see if  a suitable modification of our arguments could be applied to study variational estimates for discrete  directional maximal operators and directional singular integrals. The latter have also attracted much attention in the continuous setting, see for instance   Demeter \cite{D}.   

Higher-dimensional versions of Theorem \ref{mainthm1} and the results from \cite{CK} are topics for further investigation. For the continuous analogues in higher dimensions  we refer to the works by Demeter \cite{D2}, di Plinio and Parissis \cite{DpP}, and the  references therein.

\section{Proof of Theorem \ref{mainthm1}}
First let us say a few words on the notation. We write $e(t)=e^{2\pi i t}$.  For a finitely supported function $f$ on $\Z^d$ we define the Fourier transform 
$$\widehat{f}(\beta) = \sum_{n\in \Z^d}f(n)e(-\beta\cdot  n),$$
taking values in $\mathbb{T}^d$. Its inverse  will be denoted by 
$$\mathcal{F}^{-1}{g}(n) = \int_{\mathbb{T}^d}g(\beta)e(\beta \cdot  n)d\beta.$$
The Fourier transform on $\R$ and its inverse will be defined with the same symbols and   we will mention their use explicitly.  We will also silently use the standard identification of  functions on $\mathbb{T}^d$ with one-periodic function on $\R^d$, i.e.  functions  $g:\R^d\rightarrow \C$ satisfying $g(x+y)=g(x)$ for all $y\in \Z^d$.   
 We write $A\lesssim B$ if there exists an absolute constant $C$ such that  $A\leq CB$. If the constant depends on parameters $p_1,\ldots, p_n$ we denote that with a subscript, such as $A\lesssim _{p_1,\ldots, p_n}B$. We write $A\sim B$ if both $A\lesssim B$ and $B\lesssim A$.

 The first step in the proof of Theorem \ref{mainthm1} is to decompose the multipliers   of  the maximal operator.
For $\alpha\in \R$ we define
\begin{align*}
m_k(\alpha)&=\sum_{p\in \mathbb{P}} e(-p\alpha)2^{-k}\phi(2^{-k}p) \log p. 
\end{align*}
Then we may view $A^{\mathbb{P}}_Vf$ as a maximal operator  with the multipliers given by $\{m_k(v \cdot                                                                                                                                                                                                                                \beta )\}_{v\in V}$,
\begin{align*}
A^{\mathbb{P}}_Vf(x)= \sup_{k\geq k_V}  \sup_{v\in V} \big |\mathcal{F}^{-1}(m_k(v\cdot )\widehat{f}) (x) \big | .
\end{align*}

For the decomposition of the multipliers we need to define several auxilliary functions. Let $\chi:\R\rightarrow \R$ be a non-negative smooth function which is supported in $|x|\leq 1/2$ and equals $1$ on $|x|\leq 1/4$. 
Then we set
$$\chi_s(\alpha)=\chi(2^{10(s+4)}\alpha).$$
For $\alpha\in \R$ we define
\begin{align*}
V_k(\alpha) = \int_{\R}e(2^kt \alpha) \phi(t)dt.
\end{align*}
By $\mu(q)$ and $\varphi(q)$ we denote the M\"obius and totient functions, respectively.  A key estimate on the totient function is 
\begin{align}
\label{totient}
\varphi(q) \gtrsim_\delta q^{1-\delta},
\end{align}
valid for any $\delta>0$. 
Finally, we define  $\mathcal{R}_0 = \{0\}$ and 
for $s\in \mathbb{N}$ we  set
$$\mathcal{R}_s=\{a/q \in \mathbb{T}\cap \mathbb{Q} : 2^s\leq q < 2^{s+1}, \, (a,q)=1\}.$$
We recall the following approximation result from \cite{MTZK}.
\begin{lemma} [from \cite{MTZK}]
\label{lemmaapprox}
Let $k\geq 0$ be an integer. For any $D>2^4$ there is a constant $C=C(D)$ such that for   $\alpha\in \mathbb{R}$ one can approximate 
$$m_k (\alpha) = L_k(\alpha) + E_k(\alpha)$$
where 
\begin{align*}
L_k(\alpha) = \sum_{s\in \mathbb{N}_0} L_{k,s}(\alpha),\quad L_{k,s}(\alpha) =   \sum_{a/q\in \mathcal{R}_s}\frac{\mu(q)}{\varphi(q)}V_k(\alpha-a/q)\chi_s(\alpha-a/q),
\end{align*}
and   $E_k$    satisfies the   pointwise  estimate  
\begin{align}
\label{decayEk}
|E_k|\leq C k^{-D/8} .
\end{align}
\end{lemma} 
Note that $L_k$ is a periodic function and hence can be viewed as a function on $\mathbb{T}$.
Observe that due to the support of $\chi_s$, for each fixed $\alpha$,  the sum defining $L_{k,s}$ is non-zero for at most one term.

The proof of  Lemma \ref{lemmaapprox} proceeds by the Hardy-Littlewood circle method.
The key step is  in  establishing that the multipliers $m_k$ satisfy 
\begin{align*}
\Big | m_k(\alpha) -\frac{\mu(q)}{\varphi(q)}V_k(\alpha-a/q)\Big | &\lesssim_D\,  k^{-D/8}  \quad \textup{if} \quad  \alpha\in \mathcal{M}_k^D(a/q)\cap \mathcal{M}_k^D, \\
|m_k(\alpha)| & \lesssim_D \, k^{-D/8} \quad \textup{if}\quad  \alpha\in \mathbb{T}\setminus \mathcal{M}_k^D,
\end{align*}
where  for $D>0$ and $k\in \mathbb{N}$ the major arcs are defined by 
 \begin{align*}
 \mathcal{M}_k^D = \bigcup_{1\leq q\leq k^D} \bigcup_{a\leq q: (a,q)=1}\mathcal{M}_k^D(a/q), \quad \textup{where} \quad  \mathcal{M}_k^D(a/q) = \{\alpha\in \mathbb{T} : |\alpha-a/q| \leq 2^{-k}k^D\},
 \end{align*}
 and the set $\mathbb{T}\setminus \mathcal{M}_k^D$ is    the minor arc. These two estimates follow from Proposition 3.1 and Proposition 3.2 in \cite{MTZK}, respectively.  Theorem 2 in \cite{MTZK} then yields the result claimed in  Lemma \ref{lemmaapprox}.
  
In all of  the following we fix $\epsilon>0$, $N>0$. We also   assume $k\geq k_V$.  
Because of Lemma \ref{lemmaapprox} it suffices to prove $\ell^2$ bounds for  the operators associated with the multipliers $L_k$ and $E_k$. 
First we prove the desired bound for the error term $E_k$. This follows from   
$$\Big \| \Big ( \sum_{v\in V} \sup_{k\geq k_V} |\F^{-1}(E_k(v\cdot )\widehat{f})|^2 \Big )^{1/2}\Big \|_{\ell^2(\Z^2)} \lesssim k_V^{-1} N^{1/2}  \|f\|_{\ell^2(\Z^2)} \lesssim_\epsilon N^\epsilon \|f\|_{\ell^2(\Z^2)},$$
where we have used Plancherel and \eqref{decayEk}. 
Therefore, in order   to prove Theorem \ref{mainthm1}  
it suffices to obtain $\ell^2$ bounds for 
\begin{align}\label{Lvf}
 \mathcal{L}_Vf= \sup_{k\geq k_V} \sup_{v\in V} \big | \F^{-1}(L_{k} (v\cdot )\widehat{f})  \big |.
\end{align}

Before estimating $\mathcal{L}_Vf$ we first make   a simple   preliminary reduction. 
The following lemma tells us that when $s$ is small, we can replace the localization $\chi_s$ with the improved localization $\chi_{k_V/100}$.
\begin{lemma}
For $\alpha\in \R$, $s\in \mathbb{N}_0$ and an integer $k\geq 0$ define
\begin{align*}
L_{k,s}'(\alpha) =   \sum_{a/q\in \mathcal{R}_s}\frac{\mu(q)}{\varphi(q)}V_k(\alpha-a/q)\chi_{k_V/100}(\alpha-a/q).
\end{align*}
Then we have
\begin{align*}
\sum_{s \leq \epsilon \log N} \| \sup_{k\geq k_V} \sup_{v\in V} \big | \F^{-1}((L_{k} -L'_k)(v\cdot )\widehat{f})  \big | \|_{\ell^2(\Z^2)} \lesssim N^\epsilon  \|f\|_{\ell^2(\Z^2)}.
\end{align*}
\end{lemma}

\begin{proof}
Note that $\chi_s-{\chi}_{k_V/100}$ vanishes on $|\alpha|\lesssim 2^{-k/10}$ for any $k\geq k_V$ provided $C_1$ from Theorem \ref{mainthm1} is sufficiently large. Integrating by parts we show the decay estimate 
\begin{align*}
|V_k(\alpha)| \lesssim (2^k|\alpha|)^{-1} \lesssim 2^{-k/2}.
\end{align*}
By the  lower  bound on the totient function \eqref{totient} and Plancherel this immediately gives  
\begin{align*}
\sum_{s\leq \epsilon \log N} 2^{s(\delta-1)} \Big \|  \sum_{2^{s}\leq q < 2^{s+1}} \Big(  &  \sum_{k\ge k_V}  \sum_{v\in V}  \Big | \F^{-1}\Big(  \sum_{a:a/q\in \mathcal{R}_s} V_k(v\cdot \beta -a/q) \\
&  (\chi_s- {\chi}_{k_V/100})(v\cdot \beta-a/q) \widehat{f}(\beta) \Big) \Big| ^2 \Big )^{1/2} \Big\|_{\ell^2(\Z^2)} \lesssim  N^{C\epsilon} \|f\|_{\ell^2(\Z^2)},
\end{align*}
for an absolute constant $C>0$, which implies the claim.
\end{proof}

The next step in estimating $ \mathcal{L}_Vf$ given in \eqref{Lvf} is to  split the function $f$ into a low and high frequency part.
Let $\vartheta$ be a   function  which is  smooth and supported  in  $\{\alpha\in \R^2: |\alpha|\leq A^{-2}\}$, where $A$ is the constant from  Theorem \ref{mainthm1}. Writing $f= f_1+f_2$, where we have set
\begin{align*}
 \widehat{f_1} =\widehat{f} \vartheta,\quad \widehat{f_2}=\widehat{f}(1-\vartheta ), 
\end{align*}
we split accordingly 
\begin{align}
\label{lowhigh}
\mathcal{L}_Vf \leq \mathcal{L}_Vf_1+ \mathcal{L}_Vf_2.
\end{align}

We now briefly describe the general approach that we will use to estimate $\mathcal{L}_Vf_1$ and $\mathcal{L}_Vf_2$.   
To deal with the high frequency part we will use  the following crucial incidence estimate from \cite{CK}. Let $s\geq 1$ be an integer and $C_1$ an integer which is assumed to be sufficiently large. For $2^{s}\leq r <2^{s+1}$ and $v\in V$ we  consider the sets
\begin{align*}
K_{r,s,v} = \{\beta \in \mathbb{T}^2, |\beta|\geq A^{-2}  : |v\cdot \beta - b/r-m|\leq 2^{-C_1s} \textup{ for some } 0\leq b\leq r, \,m\in \Z\}. 
\end{align*}
These sets arise naturally from the exponential sums which will appear in the multiplier approximations. 
They consist of equally spaced tubes of thickness about $2^{-C_1 s}A$ which are perpendicular to some $v\in V$, outside a ball of radius    $A^{-2}$. The following lemma measures the overlap of these tubes outside this ball.   
\begin{lemma}[\cite{CK}]
\label{prop:incidence}
Let $\epsilon >0$, $N>0$.  Let $s\geq 1$ be an integer.
 There exist constants $C_\epsilon$ and $C_0$ depending on $\epsilon$ so that the following holds. For any $A\geq N^{C_0}$
there exists a set of vectors $V\subseteq \{x\in \Z^2: \frac{1}{100} A\leq |x| \leq  100A\}$ with   $|V|=N$  so that, such that for each $N^{\epsilon}\leq 2^s\leq N^{1/\epsilon}$ one has 
$$\Big \| \sum_{2^{s}\leq r < 2^{s+1}}\sum_{v\in V} \mathbf{1}_{K_{r,s,v}} \Big \|_{\L^\infty(\mathbb{T}^2)} \lesssim_\epsilon N^\epsilon$$
provided $C_1=C_1(A)$ in the definition of $K_{r,s,v}$ is chosen sufficiently large.
\end{lemma}

The construction of the set  $V$ from Lemma \ref{prop:incidence}  is done by choosing vectors in different directions, such that after an appropriate rescaling their components have a prime factorization structure that exhibit an intermediate amount of randomness. An upper bound on the randomness allows for a certain degree of compatibility between the arithmetic structure between vector coordinates, which  gives rise to the appearance of long arithmetic progressions of large prime gap length in the corresponding intersecting sets $K_{r,s, v}$. A lower bound on the randomness allows for these long arithmetic progressions to have different prime spacings, making it very difficult for a point to lie in the simultaneous intersection of these sets.  We refer to Appendix \ref{appendix:incidence} for a more detailed summary of the construction from \cite{CK}.

To  deal with low frequencies we  will   use the following transference   lemma by  Bourgain  \cite{Bou2}. It says that if one has a maximal multiplier operator with multipliers supported on a fundamental domain of $\mathbb{T}^d$, which is bounded on $\L^2(\R^d)$, then the maximal operator corresponding to the periodization of these multipliers is bounded on $\ell^2(\Z^d)$.
\begin{lemma} [\cite{Bou2}]
\label{lemma:bourgain}
Suppose $\{\phi_k: k \geq 1\}$ is a countable family of functions on $\R^d$ supported in $[-1/2,1/2)^d$ with $\sup_k |\phi_k|\leq 1$. Further, suppose that for 
$$Mf(x)=\sup_k|\F^{-1}(\phi_k \widehat{f}) (x)|$$
there exists a constant $B$ such that for all $f\in \L^2(\R^d)$ one has 
\begin{align}\label{conthyp}\|Mf\|_{\L^2(\R^d)}\leq B \|f\|_{\L^2(\R^d)}\end{align}
Then  there exists a constant $C$ such that for all $f\in \ell^2(\Z^d)$ one has
\begin{align*}
\|Mf\|_{\ell^2(\mathbb{Z}^d)}\leq C B \|f\|_{\ell^2(\mathbb{Z}^d)}
\end{align*}
\end{lemma}
Here the use of the symbols for the Fourier transform should be interpreted either on $\R$ or $\Z$, respectively.

With these  results in mind  we turn to estimating $\mathcal{L}_Vf_1$ and $\mathcal{L}_Vf_2$ in \eqref{lowhigh}.

 \subsection{High frequency term $\mathcal{L}_Vf_2$} 
Applying the triangle inequality and use the the lower  bound on the totient function \eqref{totient}  it suffices to estimate the $\ell^2$ norm of
\begin{align}
  \sum_{s\in \mathbb{N}_0} 2^{s(\delta-1)} \sum_{2^{s}\leq q < 2^{s+1}} \sup_{k\ge k_V} \sup_{v\in V} \Big | \F^{-1}\Big(  \sum_{a:a/q\in \mathcal{R}_s}  V_k(v\cdot \beta -a/q)
\chi_{k_0}(v\cdot \beta-a/q) \widehat{f}_2(\beta) \Big)\Big|, \label{toestimate}
\end{align}
where $k_0=k_0(s)$ is given by
\begin{align*}
k_0(s) =\left \{ 
\begin{array}{lll}
k_V ,& s\leq \epsilon \log N\\
s ,&  s> \epsilon \log N .
\end{array}
\right .
\end{align*}
Replacing the supremum in $v$ by a square sum it suffices to estimate the $\ell^2$ norm of
\begin{align}\nonumber
 \sum_{s\in \N_0} 2^{s(\delta-1)} \sum_{2^{s}\leq q < 2^{s+1}} \Big ( \sum_{v\in V } \Big( \sup_{k\ge k_V}  \Big | \F^{-1}\Big( \sum_{a:a/q\in \mathcal{R}_s} & V_k(v \cdot \beta-a/q)\\\label{est1}
 & \chi_{k_0}(v\cdot \beta-a/q) \widehat{f}_2(\beta) \Big)\Big|\Big) ^2 \Big)^{1/2}.
\end{align}

Since for a  fixed  $2^{s}\leq q<2^{s+1}$ the support of $\chi_{k_0}$ is much smaller than the size of $q$, and $\chi_{k_0} = \chi_{k_0}\chi_{k_0-1}$, the last display can be identified with 
\begin{align}\nonumber
 \sum_{s\in \N_0} 2^{s(\delta-1)} \sum_{2^{s}\leq q < 2^{s+1}} \Big ( \sum_{v\in V } \Big( \sup_{k\ge k_V}  \Big | \F^{-1}\Big( &\sum_{m\in Z} \sum_{1\leq a\leq q} V_k(v\cdot \beta -a/q-m)\\ \label{termtosum}
 &\chi_{k_0}(v\cdot \beta-a/q-m)\widehat{f_{q,v}}(\beta) \Big)\Big|\Big)^2 \Big)^{1/2},
\end{align}
where $f_{q,v}$ is defined via 
$$\widehat{f_{q,v}}(\beta) =  \sum_{m\in \Z} \sum_{b\leq q: (b,q)=1}\chi_{k_0-1}(v\cdot \beta -b/q-m) \widehat{f}_2(\beta) .  $$
The advantage of \eqref{termtosum} is that the sum runs over $a\leq q$ rather than only coprime elements, which brings in additional cancellation.     

To bound \eqref{termtosum} we first sum the corresponding multiplier over $a\leq q$.
For fixed $k,q$ and $v$ we consider
\begin{align}\label{multiplieraq}
\F^{-1}\Big(  \sum_{1\leq a\leq q} \phi_{k,k_0} (v\cdot \beta-a/q) \widehat{f_{q,v}} (\beta)\Big )  
\end{align}
with  $\phi_{k,k_0} =  \sum_{m\in \Z}V_k(\cdot-m)  \chi_{k_0}(\cdot -m)$.
By Fourier inversion on $\phi_{k,k_0}$ we have  
\begin{align} \label{inverseft}
\sum_{1\leq a\leq q} \phi_{k,k_0}(v\cdot \beta -a/q)=\sum_{1\leq a\leq q} \sum_{n\in \mathbb{Z}} (\F^{-1}\phi_{k,k_0} )(n) e(- v\cdot \beta n) e(na/q ).
\end{align}
Using
\begin{align*}
\sum_{1\leq a\leq q}e(na/q ) = \left\{ \begin{array}{ll}
q & n=qj\, \textup{ for some  } j\in \Z \\
0&\textup{otherwise}
\end{array} \right .
\end{align*}
and splitting the sum in $n$ in \eqref{inverseft} into the cases depending on whether $n$ is divisible by $q$ or not,  we see that \eqref{inverseft} equals  
\begin{align*}
  q \sum_{j\in \mathbb{Z}} \F^{-1}\phi_{k,k_0} (jq) e(- v\cdot \beta jq).
\end{align*}
Therefore, \eqref{multiplieraq} can be written as 
\begin{align*}
 \mathcal{F}^{-1}({\phi_{k,k_0,q}}(qv\cdot \beta) \widehat{f_{q,v}}(\beta)) = \sum_{n\in \mathbb{Z}}f_{q,v}(x-n(qv)) \phi_{k,k_0,q}(n ),
\end{align*}
where we have defined $\phi_{k,k_0,q}$ via $\F^{-1}\phi_{k,k_0,q}(n) = q\F^{-1}\phi_{k,k_0}(qn)$.  

Next we would like to bound these operators for each fixed $q$ and $v$.  
For this we first show that the multipliers $\{{\phi_{k,k_0,q}}\}_{k\geq 1}$ are   bounded uniformly in $k$ and $q$. This follows from
\begin{align*}
|{{\phi}_{k,k_0,q}}| \leq \| \F^{-1}\phi_{k,k_0,q}\|_{\ell^1(\Z)} = \|q\F^{-1}(V_k\chi_{k_0})(qn) \|_{\ell^1(n)},
\end{align*}
and  the right-hand side is further dominated   by
\begin{align*}
q 2^{-\max(k,k_0)} \|(2^{-\max(k,k_0)}qn)^{-3}\|_{\ell^1(n)} \lesssim q 2^{-\max(k,k_0)} \int_{\R} (1+|2^{-\max(k,k_0)}qt|)^{-2} dt \lesssim 1.
\end{align*}
Using boundedness of the continuous Hardy-Littlewood maximal operator 
and Bourgain's transference   from Lemma \ref{lemma:bourgain} applied to the corresponding multipliers in $[-1/2,1/2)$ we obtain  for  a function $g$ on $\Z$
\begin{align*}
  \| \sup_k  | \mathcal{F}^{-1}({\phi_{k,k_0,q}}(\beta) \widehat{g}(\beta))| \|_{\ell^2(\mathbb{Z})}  \lesssim \|g\|_{\ell^2(\mathbb{Z})}.
\end{align*}
By  Lemma \ref{lemma:2dto1d} stated in proven in Appendix \ref{appendix:transf}, we transfer this result to directional averages on $\Z^2$ and   obtain 
  uniform in $v$ boundedness of the convolution operator 
\begin{align*}
\sup_{v \in \mathbb{Z}^2} \Big \| \sup_k \Big |  \sum_{n\in \mathbb{Z}}f_{q,v}(x-n(qv))  \F^{-1} \phi_{k,k_0,q}(n)  \Big |  \Big  \|_{\ell^2(\mathbb{Z}^2)} \lesssim \|f_{q,v}\|_{\ell^2(\mathbb{Z}^2)}.
\end{align*}

Returning to the expression \eqref{est1}, we see that we have reduced the problem to estimating 
\begin{align}\label{finalest}
\sum_{s\in \mathbb{N}_0}2^{s(\delta-1)}\sum_{2^s\le q<2^{s+1}} \Big(\sum_{v\in V}\|f_{q, v}\|^2_{\ell^2(\Z^2)}\Big)^{1/2}.
\end{align}
Now we
 split the sum in $s$ in three regimes. If $s\leq \epsilon \log N$  we use Plancherel, which leaves us with having to bound
 \begin{align*}
\sum_{s\leq  \epsilon \log N } 2^{s(\delta-1)} \sum_{2^s\leq q < 2^{s+1}}\;  \sup_{|\beta|\geq A^{-2}} \Big( \sum_{v\in V}  \mathbf{1}_{K_{q2^{\log N-s},\log N,v}} \,  \Big )^{1/2} \|f\|_{\ell^2(\Z^2)}.
 \end{align*}
 Using   the incidence estimate from  Lemma \ref{prop:incidence} applied to the quantity in the bracket and the trivial estimates for the sum in $q$ and $s$ we obtain a bound of the last display by 
 $$\lesssim_\epsilon N^{C\epsilon}\|f\|_{\ell^2(\Z^2)}$$ where $C$ is an absolute constant. 
  We remark that since the maximal operator is not sensitive to modulations of the function $f$, in this case one could alternatively pass to the spatial side to obtain $N^\epsilon$ copies of the   problem for $s=0$, to which one then applies the incidence estimate. 
 
 If $s>\epsilon \log N$,  we first apply  Cauchy-Schwarz in $q$. 
This dominates \eqref{finalest} by 
\begin{align*}
\sum_{s>\epsilon\log N}2^{s(\delta-1/2)}\bigg(\sum_{2^s\le q<2^{s+1}}\sum_{v\in V}\|f_{q, v}\|_{\ell^2(\Z^2)}^2\bigg)^{1/2},
\end{align*}
which is further bounded by
 \begin{align}\label{boundincidence}
\sum_{s> \epsilon \log N } 2^{s(\delta-1/2)} \sup_{|\beta|\geq A^{-2}} \Big( \sum_{2^s\leq q < 2^{s+1}} \sum_{v\in V}  \mathbf{1}_{K'_{q,s,v}} \,  \Big )^{1/2}\|f\|_{\ell^2(\Z^2)}.
 \end{align}
 Here $K'_{q,s,v}$ is defined analogously to $K_{q,s,v}$, but only reduced rationals are considered, i.e.
 \begin{align*}
K'_{q,s,v} = \{\beta \in \mathbb{T}^2, |\beta|\geq A^{-2}  : |v\cdot \beta - b/q-m|\leq 2^{-C_1s} \textup{ for some } b\leq q, (b,q)=1,\, m\in \Z\}. 
\end{align*}
 
If $\epsilon \log N \leq s < \epsilon^{-1} \log N$ we use the incidence estimate in Lemma \ref{prop:incidence}  and sum a geometric series in $s$ to obtain a bound of \eqref{boundincidence} by 
$$\lesssim_\epsilon N^\epsilon\|f\|_{\ell^2(\Z^2)}.$$

Finally, if  $\epsilon^{-1} \log N \leq s$ we use the fact that for a  fixed $v$ one has
\begin{align*}
\sum_{2^s\leq q < 2^{s+1}} \mathbf{1}_{K_{q,s,v}'}\leq 1.
\end{align*}
Therefore, we may estimate \eqref{boundincidence} by
 \begin{align*}
N^{1/2}   \|f\|_{\ell^2(\Z^2)} \sum_{s> \epsilon^{-1} \log N} 2^{s(\delta-1/2)} \lesssim N^{C \epsilon}  \|f\|_{\ell^2(\Z^2)}.
\end{align*}
This finishes the proof for the high frequency term.

\subsection{Low frequency term $\mathcal{L}_Vf_1$}
 First we consider the case $\epsilon^{-1} \log N < s$. 
 Performing the analogous steps from \eqref{toestimate} to \eqref{finalest}
we see that we have reduced the problem to estimating   
\begin{align} \label{aftercs}
\sum_{s\in \mathbb{N}_0}2^{s(\delta-1)}\sum_{2^s\le q<2^{s+1}} \Big(\sum_{v\in V}\|\widetilde{f}_{q, v}\|^2_{\ell^2(\Z^2)}\Big)^{1/2},
\end{align}
where 
$$\widehat{\widetilde{f}_{q,v}}(\beta) =  \sum_{m\in \Z} \sum_{b\leq q: (b,q)=1}\chi_{k_0-1}(v\cdot \beta -b/q-m) \widehat{f}_1(\beta) .  $$
 By Cauchy-Schwarz in $q$ and Plancherel we bound \eqref{aftercs} by 
$$ \sum_{s> \epsilon^{-1} \log N } 2^{s(\delta-1/2)} \sup_{|\beta|\leq A^{-2}} \Big( \sum_{2^s\leq q < 2^{s+1}} \sum_{v\in V}  \mathbf{1}_{K''_{q,s,v}} \,  \Big )^{1/2}\|f\|_{\ell^2(\Z^2)},$$
where 
\begin{align*}
K''_{q,s,v} = \{|\beta|\leq A^{-2} : |v\cdot \beta - b/q-m|\leq 2^{-C_1s} \textup{ for some }  b\leq q, (b,q)=1,\, m\in \Z\}. 
\end{align*}
Therefore, we may estimate \eqref{aftercs} by
 \begin{align*}
N  \|f\|_{\ell^2(\Z^2)} \sum_{s> \epsilon^{-1} \log N} 2^{s(\delta-1/2)}  \lesssim N^{C \epsilon}  \|f\|_{\ell^2(\Z^2)}.
\end{align*}

It remains to consider the case $s\leq \epsilon^{-1}\log N$. By the lower bound on the totient function \eqref{totient}, it suffices to bound the $\ell^2(\mathbb{Z}^2)$ norm of
\begin{align}\nonumber
\sum_{s\le\epsilon^{-1}\log N}2^{s(\delta-1)}\sup_{v\in V}\sup_{k\ge k_V} \Big|\mathcal{F}^{-1}\Big(\sum_{2^s\le q<2^{s+1}}\sum_{a:\,a/q\in\mathcal{R}_s}&V_k(v\cdot\beta-a/q)\\ \label{second2}
&\chi_{k_0}(v\cdot\beta-a/q)
\widehat{\vartheta}(\beta)\widehat{f}(\beta)\Big)\Big|.
\end{align}
Assuming  $A$ is  larger than $N^{100/\epsilon}$, it follows that for a fixed $s\le\epsilon^{-1}\log N$, the ball of radius $A^{-2}$ in the torus intersects the support of  $V_k(v\cdot\beta-a/q)\chi_{k_0}(v\cdot\beta-a/q) $
for at most $4$ different values of $a/q\in\mathcal{R}_s$. Indeed, this observation follows from the fact that the spacing between tubes is much larger than $A^{-2}$. We may thus conclude that the multiplier corresponding to \eqref{second2} for a fixed $s$, i.e.
\begin{align*}
\Big( \sum_{2^s\le q<2^{s+1}}\sum_{a:\,a/q\in\mathcal{R}_s}V_k(v\cdot\beta-a/q)\chi_{k_0}(v\cdot\beta-a/q)  \Big )\widehat{\vartheta}(\beta),
\end{align*}
is bounded uniformly in $k$ and $v$. 
We would now like to apply Lemma \ref{lemma:bourgain} to each term in \eqref{second2} with $s$ fixed. The hypothesis \eqref{conthyp} is then satisfied with $B=N^{\epsilon}$ and the Euclidean maximal operator $M$ given by the Euclidean directional maximal function, which satisfies the bound by \cite{Katz1}. Hence Lemma \ref{lemma:bourgain} implies that the  $\ell^2(\mathbb{Z}^2)$ norm of \eqref{second2} is bounded by
\begin{align*}
\lesssim_\epsilon N^{\epsilon}\|f\|_{\ell^2(\mathbb{Z}^2)}
\end{align*}
as desired. This finishes the proof.


\appendix

\section{Transference argument}
\label{appendix:transf}

 \begin{lemma}
\label{lemma:2dto1d}
Assume that $\{\phi_k\}_{k\in \Z}$ is a sequence of   functions on $\mathbb{T}$ such that there is a constant $B$ such that for any   $f$ on $\mathbb{Z}$ one has 
 \begin{align}
 \label{opZ}
\big \| \sup_{k}    \big |\F^{-1}(\phi_k\widehat{f}) (x) \big | \big \|_{\ell^2(\Z)}\leq B \|f\|_{\ell^2(\mathbb{Z})} .
\end{align}
Then there is a constant $C$ such that for any   $f$ on $\mathbb{Z}^2$ and any  $0\neq v\in \mathbb{Z}^2$  one has
 \begin{align}
 \label{opZ2}
\big \| \sup_{k}    \big |\F^{-1}(\phi_k(v\cdot  )\widehat{f})(x) \big | \big \|_{\ell^2(\Z^2)}\leq CB  \|f\|_{\ell^2(\mathbb{Z}^2)} .
\end{align}
\end{lemma}
\begin{proof}
 Passing to the spatial side,  on the left-hand side of  \eqref{opZ2} we have the norm of 
$$Tf(x) = \sup_k \Big |\sum_{n\in \Z}f(x-nv)\F^{-1}{\phi_k}(n) \Big |.$$
Write 
$$\Z^2 = \bigcup_{x\in \mathbb{Z}^2}\{x+ nv : n\in \mathbb{Z}\}=\bigcup_y L_y,$$
where $L_y$ are disjoint collections of points indexed by conjugacy classes of  $\{ y + nv \}$ (i.e. we identify two lines that  overlap).
Observe that
$$T(f\mathbf{1}_{L_y}) = \mathbf{1}_{L_y}T(f\mathbf{1}_{L_y}),$$
so by sublinearity we may estimate 
$$ Tf = T\Big ( \sum_y f \mathbf{1}_{L_y} \Big ) \leq \sum_y \mathbf{1}_{L_y}
T( f \mathbf{1}_{L_y}).$$
Upon taking $\ell^2$ norms and by disjointness of the lines $L_y$ we have that
\begin{align}
\label{suminy}
\|Tf\|_{\ell^2(\Z^2)}^2 \leq \sum_{y} \|\mathbf{1}_{L_y} T(f\mathbf{1}_{L_y})\|^2_{\ell^2(\Z^2)} = \sum_y \sum_{x\in L_y} \Big | \sup_k \Big | \sum_{n\in \Z} f(x-n v)\F^{-1}\phi_k(n) \Big | \Big |^2 .
\end{align}
For a fixed $y\in \Z^2$ define 
 a function $g_y$ on $\Z^2$ by setting 
 \begin{align*}
 g_y(z) =  \left\{ \begin{array}{ll}
 f(y-a v )&: z = y-(a,0), \, a\in \Z \\
 0 &: \textup{otherwise}
 \end{array}
  \right . .
 \end{align*}
Every $x\in \L_y$ can be written as $x=y-n'v$ for some $n'\in \Z$. 
 Then, for a fixed $y$, the right-hand side of \eqref{suminy} equals  (writing $y=(y_1,y_2)$)
 $$\sum_{n'\in \Z} \Big | \sup_k \Big | \sum_{n\in \Z} g_y(y_1-n-n', y_2) \F^{-1}\phi_k(n) \Big | \Big |^2 \lesssim \sum_{n'\in \Z}| g_y(y_1-n',y_2)|^2,$$
 where we have used the inequality \eqref{opZ}   on the one-dimensional function $g_y(\cdot, y_2)$. Summing in $y$  we obtain for the last display  
 \begin{align*}
 \sum_y\sum_{n'\in\Z}| g_y(y_1-n',y_2)|^2 = \sum_{y} \sum_{n'\in \Z} |f(y-n'v)|^2 = \sum_{x\in \Z^2} |f(x)|^2 = \|f\|^2_{\ell^2(\Z^2)},
 \end{align*}
 which is the desired estimate.
\end{proof}

\section{Proof of Lemma \ref{prop:incidence}}
\label{appendix:incidence}
In this   section we provide a  short summary of the proof of Lemma \ref{prop:incidence}, which originated in \cite{CK}. For more details we refer the reader to \cite{CK}.

To prove Lemma \ref{prop:incidence} it suffices to choose a collection of vectors $\widetilde{V}=\{v_1, v_2, \ldots, v_N\}$ contained  in the   annulus $\{x\in \R^2: \frac{1}{100} \leq  |x| \leq 100\}$, so that  for
$$V=A\widetilde{V}\subseteq\{x\in \Z^2: \frac{1}{100}A \leq  |x| \leq 100A\}$$ the following holds. For any subcollection $S\subseteq \widetilde{V}$ with $|S|=N^{\epsilon}$  and any choice of integers $\{r_{v}\}_{v\in S}$ with $2^{s}\le r_{v}<2^{s+1}$, 
no point   is contained in 
\[ \bigcap_{r_{v} :\, v \in S} \widetilde{K}_{r_{v}, s, v}.\] 
Here we have defined 
\[ \widetilde{K}_{r, s, v} =  \{ \beta \in A\mathbb{T}^2 :  |\beta| \geq A^{-1}, \  | v \cdot \beta - \frac{b}{r}| \leq 2^{-C_1 s} \text{ for some } b \lesssim Ar, b \in \Z  \} 
\]
 for $2^{s}\leq r < 2^{s+1}$. Note that we have  restricted to $b\lesssim Ar$ due to periodicity. 
Recall that $C_1=C_1(A)$ is a large constant, and so the thickness of the tubes is very small comparable to their spacing.

Let $M$ be a sufficiently large integer which will be determined later.
 Denote by $P_M$ the smallest $N^{\epsilon /2}$ primes in 
$[N^{M/\epsilon}, 10N^{M/\epsilon}]$.
By Stirling's approximation
we can find an integer  $\kappa \sim  2\epsilon^{-1}$ with
\begin{align*}
{ N^{\epsilon/2} \choose \kappa}\sim_\epsilon N.
\end{align*}
By adjusting our choice of $\kappa$ and replacing $\epsilon$ by a suitable approximate   we may assume that $N^{M\kappa/\epsilon}$ is an integer. 
 We define $\widetilde{V}=\{v_1, \ldots, v_N\}$ component-wise by setting for $1\leq i \leq N$
\begin{align*}
(v_i)_x & =m_i \,Q_i \,N^{-M\kappa/\epsilon}\, p_{i_1}p_{i_2}\cdots p_{i_{\kappa}},\\
(v_i)_y&=n_i\, Q_i\, N^{-M\kappa/\epsilon}\, p_{i_1}p_{i_2}\cdots p_{i_{\kappa}},
\end{align*}
 with the following constraints.
\begin{itemize}
\item For all $i$, we have $1/4\le n_i/m_i\le 1/2$ with $n_i, m_i>0$.
\item For all $i$, we have $(m_i, n_i)\in  \{x\in \Z^2:\,\frac{1}{10} N^2\leq |x| \leq 10N^2 \}$. 
\item For all $i\ne j$, we have $(m_i, n_i)^{\perp}\cdot (m_j, n_j)\ne 0$.
\item   For all $1\leq j, j' \leq \kappa$, $j\neq j'$, we have   $p_{i_j}\neq p_{i_{j'}}$. 
\item No two $v_i$ have the exact same collection of corresponding primes $\{p_{i_1}, p_{i_2}, \ldots, p_{i_{\kappa}}\}$.
\item Each $Q_i$ is a dyadic number with $ 2^{-100\kappa}N^{-2}\le Q_i\le 2^{100\kappa} N^{-2} $, chosen so that $\frac{1}{10}\leq |v_i|\leq 10$.
\end{itemize}
Note that for any $A\geq N^{2M\kappa/\epsilon}$   we can find $\widetilde{A}$ which is an integer multiple of 
$$N^{M\kappa/\epsilon} \prod_{i: Q_i \leq 1} Q_i^{-1}$$ such that  $\frac{1}{10}A \leq \widetilde{A} \leq 10A$. Then the rescaled vectors $\widetilde{A}v_i$ are integers and satisfy  $\frac{1}{100}A\leq \widetilde{A}|v_i| \leq 100 A$. 

Next we show that the above defined set  of vectors yields the conclusion from Lemma \ref{prop:incidence}. 
By   construction, the  angle between any two distinct vectors in $\widetilde{V}$  is at least $\gtrsim 1/N^{2}$. Therefore, taking $C_1$ sufficiently large, the intersection of any two sets  $\widetilde{K}_{r_1,s, v_1} \cap \widetilde{K}_{r  _2,s, v_2}$ belongs to an $N^{-C_1/2}$-neighborhood of a grid, which is  a lattice with generators parallel to  $v_1^{\perp}$ and $v_2^{\perp}$. 
The $x$-coordinates of all points in this grid belong to an $O(N^{-C_1/2})$-neighborhood of the rank $2$ arithmetic progression 
\begin{align}\label{e:SET}
 \|v_1\|\|v_2\| |\langle v_1^{\perp}, v_2\rangle |^{-1}\big\{\frac{a}{r_1}(v_1^{\perp})_x+\frac{b}{r_2}(v_2^{\perp})_x\,; \, a, b\in\mathbb{Z}, 0\leq a, b \lesssim A^2\big\}.
\end{align}
Therefore, $(\widetilde{K}_{r_1,s, v_1}\cap \widetilde{K}_{r_2,s,v_2})_x$ is contained in an $O(N^{-C_1/2})$ neighborhood of {\eqref{e:SET}}, and we obtain an analogous result for $(\widetilde{K}_{r_1,s, v_1}\cap \widetilde{K}_{r_2,s, v_2})_y$.

Suppose now that we are given $S \subseteq \widetilde{V}$ with $|S| = N^{\epsilon}$.
By Stirling's approximation we may  choose an integer   
$K \sim N^{{\epsilon^2}/{2}}$ such that 
\begin{align*}
{K \choose \kappa}\sim N^{\epsilon/2}.
\end{align*}

\begin{lemma} 
\label{lemma:count}
For any  $S\subseteq \widetilde{V}$ with $|S|=N^\epsilon$, one may choose at least $K$ distinct primes $p_{1}, p_{2}, \ldots, p_{K} \in P_M$ so that the following holds. There exist disjoint sub-collections $S_1,\ldots, S_K$ of $S$, each of cardinality $2$, so that if $S_j = \{ v_{j_1}, v_{j_{2}}\}$, both
\begin{equation}\label{e:pdiv}
(v_{j_1})_y \, N^{M\kappa/\epsilon}\, Q_{j_1}^{-1}\, \text{ and } \, (v_{j_2})_y \, N^{M\kappa/\epsilon}\, Q_{j_2}^{-1}
\end{equation}
are divisible by $p_{j}$ 
 for some $j\leq K$, 
but there exists $S_k$, $k\neq j$, such that the corresponding two factors \eqref{e:pdiv} are not divisble by $p_j$.

\end{lemma}
The sets $S_j$ from Lemma \ref{lemma:count} are defined inductively as follows. 
 If $j\leq K$, we take $S_j$ to be any two elements not contained in $\bigcup_{j'<j}S_{j'}$ such that the  associated terms \eqref{e:pdiv} are divisible by some   $p_{j}\notin\{p_{1}, \ldots, p_{{j-1}}\}$, where $p_1,\ldots, p_{j-1}$  have been chosen   in the previous steps of the inductive process. That there are enough primes to make this is possible, it follows by a counting argument.

Given $S_1,\ldots, S_K$ from the lemma, let us denote 
 \begin{align*}
 \mathcal{K}_{j} = \big(\widetilde{K}_{r_{j_1}, s, v_{j_1}}\cap \widetilde{K}_{r_{j_2}, s,v_{j_2}}\big)_x, \quad  \mathcal{K}_{j}^{(2)} = \{\alpha^2 : \alpha\in  \mathcal{K}_{j}\}.
 \end{align*}
where  $v_{j_1},v_{j_2}$ are the two distinct elements of $S_j$, and $2^{s}\leq r_{j_1},r_{j_2} < 2^{s+1}$.
By Lemma \ref{lemma:count}, if $C_1\geq 100M\kappa/\epsilon^3$, we have
$$
\|v_{j_1}\|^{-1}\|v_{j_2}\|^{-1}   |\langle v_{j_1}^{\perp}, v_{j_2}\rangle|\, N^{M\kappa/\epsilon} \, r_{j_1}r_{j_2}\,\mathcal{K}_{j}\subseteq p_{j}\mathbb{Z}+O(N^{-C_1/4}).
$$ 
The lower bound on $C_1$ guarantees that the error term does not increase much in size.
Taking $N^{C_1}\geq A^{100}$  and squaring both sides we  obtain
$$
\|v_{j_1}\|^{-2}\|v_{j_2}\|^{-2} |\langle v_{j_1}^{\perp}, v_{j_2}\rangle|^2  N^{2M\kappa/\epsilon} \, r_{j_1}^2r_{j_2}^2\,\mathcal{K}^{(2)}_{j} \subseteq p_{j}\mathbb{Z}+O(N^{-C_1/5}).
$$ 
Since $m_{j_1}, n_{j_1}, m_{j_2}, n_{j_2}$ are integers of size approximately $N^2$, choosing  $C_1$ is large enough and expanding out the scaling factor on the left-hand  side     further yields
 $$
|m_{j_1}n_{j_2}-n_{j_1}m_{j_2}|\, N^{2M\kappa/\epsilon} \, r_{j_1}^2r_{j_2}^2\,\mathcal{K}_{j}^{(2)} \subseteq p_{j}\mathbb{Z}+O(N^{-C_1/6}).
$$
But then it follows that for  $K'\le C_1/1000$ one has
\begin{align*}
\bigcap_{j=1}^{K'} \mathcal{K}_{j}^{(2)}
\subseteq \Big (N^{2M\kappa/\epsilon}\prod_{j=1}^{K'}\Big(|m_{j_1}n_{j_2}-n_{j_1}m_{j_2}|\, r_{j_1}^2r_{j_2}^2\Big)\Big)^{-1}\, \Big(\prod_{j=1}^{K'}p_{j}\Big)\mathbb{Z}+O(N^{-C_1/9}),
\end{align*}
where $p_1,\ldots, p_{K'}$ are the distinct primes from Lemma \ref{lemma:count}. 
We estimate
$$N^{2M\kappa/\epsilon}\prod_{j=1}^{K'}\Big(|m_{j_1}n_{j_2}-n_{j_1}m_{j_2}|\, r_{j_1}^2r_{j_2}^2\Big)\lesssim N^{2M/\epsilon^3+C_1/100\epsilon}.$$
Since
$\prod_{j=1}^{K'}p_{j}\gtrsim N^{C_1M/(1000\epsilon)},$
for  $M$ sufficiently large depending on $\epsilon$ we obtain 
$$
\Big|\Big(N^{2M\kappa/\epsilon}\prod_{j=1}^{K'}\Big(|m_{j_1}n_{j_2}-n_{j_1}m_{j_2}|\, r_{j_1}^2r_{j_2}^2\Big)\Big)^{-1}\, \Big(\prod_{j=1}^{K'}p_{j}\Big)(\mathbb{Z}\setminus\{0\})\Big|\geq N^{C_1/100}.$$
But $N^{C_1/100}\geq  2^{10}A^2$ if $C_1$ is sufficiently large. Since each  of the sets $\mathcal{K}_{j}^{(2)}$  is a subset of $[-1000A^2, 1000A^2]$, 
it follows that
$$
\bigcap_{j=1}^{K'}\mathcal{K}_{j}^{(2)}\subseteq [-N^{-C_1/7}, N^{-C_1/7}].
$$
Performing  a symmetric argument in the $x$ coordinate we    obtain
$$\bigcap_{j=1}^{K'}\big(\widetilde{K}_{r_{j_1}, s,v_{j_1}}\cap \widetilde{K}_{r_{j_1},s, v_{j_2}}\big)\subseteq B_{N^{-C_1/14}}(0),$$
where $B_R(0)$ denotes the ball of radius $R$ about the origin.
But for $C_1$ sufficiently large, we have $B_{N^{-C_1/14}}(0)\subseteq B_{A^{-1}}(0)$. By definition any set  $\widetilde{K}_{r, s,v}$ has empty intersection with $B_{A^{-1}}(0)$,   so this completes the proof of Lemma \ref{prop:incidence}.


\end{document}